\documentclass[12pt]{amsart}   
\usepackage{amsmath,amsthm,amssymb,mathtools,marvosym}
\usepackage[english]{babel}
\usepackage[autostyle]{csquotes}
\newtheorem{theorem}{Theorem}[section]
\newtheorem{Lemma}[theorem]{Lemma}
\newtheorem{Proposition}[theorem]{Proposition}
\newtheorem{Corollary}[theorem]{Corollary}

\newtheorem{Example}[theorem]{Example}
\newtheorem{Remark}[theorem]{Remark}

\begin{document} 
\title[Pointwise dynamics under Orbital Convergence]{Pointwise dynamics under Orbital Convergence}   
\author[Abdul Gaffar Khan, Pramod Kumar Das, Tarun Das]{Abdul Gaffar Khan$^{1}$, Pramod Kumar Das$^{2}$ and Tarun Das$^{1}$}                 
\keywords{Expansivity, Shadowing, Transitivity, Topological Stability, Chaos \vspace*{0.08cm}\\ 
\vspace*{0.01cm}
\Letter{Tarun Das} \\
\vspace*{0.08cm}
tarukd@gmail.com \\
\vspace*{0.01cm}
Abdul Gaffar Khan \\
\vspace*{0.08cm}
gaffarkhan18@gmail.com\\
\vspace*{0.01cm}
Pramod Kumar Das\\
\vspace*{0.08cm}
pramodkumar.das@nmims.edu\\
\vspace*{0.01cm}
\textit{$^{1}$Department of Mathematics, Faculty of Mathematical Sciences, University of Delhi, Delhi, India.}\\ 
\hspace*{0.11cm}\textit{$^{2}$School of Mathematical Sciences, Narsee Monjee Institute of Management Studies, Vile Parle, Mumbai-400056,  India.} 
}

\begin{abstract}
We obtain sufficient conditions under which the limit of a sequence of functions exhibits a particular dynamical behaviour at a point like expansivity, shadowing, mixing, sensitivity and transitivity. We provide examples to show that the set of all expansive, positively expansive and sensitive points are neither open nor closed in general. We also observe that the set of all transitive and mixing points are closed but not open in general.  We give examples to show that properties like expansivity, sensitivity, shadowing, transitivity and mixing at a point need not be preserved under uniform convergence and properties like topological stability and $\alpha$-persistence at a point need not be preserved under pointwise convergence. 
\end{abstract} 
\maketitle 

\section{Introduction}

The idea of studying the behaviour of a dynamical system from pointwise viewpoint was initiated by Reddy. In the process of answering a question posed by Gottschalk to him, he introduced and studied pointwise expansivity, a strictly weaker notion than expansivity \cite{RP}. 
The power and the beauty of pointwise dynamics got highlighted in the recent works including \cite{AO, MSP, XZ}.
In \cite{AO}, Akin introduced the concept of chain continuity at a point which is a stronger version of shadowable point introduced in \cite{MSP} by Morales and proved that every chain transitive continuous map with chain continuity at a point must be equicontinuous \cite[Corollary 2.3]{AO} which is interestingly not true for chain transitive systems with shadowable points. A decade later, authors have introduced \cite{XZ} the concept of entropy point which worked as a key ingredient in the proof of \cite[Theorem 3]{MI}. In this theorem, Moothatu has proved that certain kind of continuous map with shadowing property has positive entropy. Recently, Morales has proved that unlike expansivity, a homeomorphism on a compact metric space has shadowing if and only if each point is shadowable \cite{MSP}. In \cite{KPSP}, the notion of entropy point is used by Kawaguchi to show that the existence of certain kind of e-shadowable points implies positive entropy \cite{KQS}. In \cite{DKDM}, authors have studied the relation of specification points with Devaney chaotic points and positive entropy of the system.  In the same paper, authors have provided an example of pointwise measure expansive homeomorphism which is not measure expansive. They have also proved that mixing at a shadowable point is not sufficient for it to be a specification point, but mixing of the map forces a shadowable point to be a specification point. Koo et. al. have recently studied the connection of shadowable points with topologically stable points and $\alpha$-persistent points in \cite{KLMP}. 
\medskip

Study of a dynamical system deals with the behaviour of an individual orbit but it is not always possible to track down the real behaviour of each orbit and here, the role of predicting the nature of an orbit via approximating it by a sequence of points (pseudo-orbits) or functions comes into picture. Also under natural constraints, the mathematical modelling of a system induces a discrete or continuous system as an approximation of the original system. Thus, a natural question arises is to study  the degree of closeness of dynamical behaviour of approximated system and the original system. Such questions have also been considered in \cite{AAU}, where author has proved that a uniform limit of transitive maps is transitive \cite[Theorem 3.1]{AAU}. Unfortunately, in an erratum \cite{AMPE}, authors gave a counter example to disprove this result. Also \cite[Example 5]{YZZD} disproved \cite[Theorem 3.2]{AAU}. Various sufficient conditions for transitivity, mixing, sensitivity etc. of limit of a sequence of functions have been studied in \cite{FDA, LA, SU, YZZD}. 
\medskip

This paper is distributed as follows. Section 2 contains preliminaries required for the remaining. In Section 3, we provide sufficient conditions under which the limit of a sequence of functions exhibits particular dynamical behaviour at a point like expansivity, $\mu$-expansivity, transitivity, mixing, Devaney chaos, shadowing, specification, topological stability and $\alpha$-persistence. 
In Section 4, we discuss topological nature of the set of all points with particular dynamical property like expansivity, sensitivity, transitivity and mixing. We provide examples to show that properties like expansivity, sensitivity, shadowing, transitivity and mixing at a point need not be preserved under uniform convergence and properties like topological stability and $\alpha$-persistence at a point need not be preserved under pointwise convergence. 

\section{Preliminaries}
Throughout this paper, $(X,d)$ denotes a metric space equipped with the metric $d$. 
We say that $X$ is Mandelkern locally compact metric space if every bounded subset of $X$ is contained in a compact set, which is equivalent to saying that every closed ball of finite radius in $X$ is compact. It is easy to check that, such spaces are complete.  We set $B(x, \epsilon) = \lbrace y\in X : d(x, y) < \epsilon\rbrace$, $B^{-}(x, \epsilon) = B(x, \epsilon)\setminus \lbrace x \rbrace$ and $B[x, \epsilon] = \lbrace y\in X : d(x, y) \leq \epsilon\rbrace$. 
\medskip

We shall consider the bounded metric on $X$ defined by $\overline{d}(x,y)= \min\lbrace d(x,y),1\rbrace$ and the metric on the set of all continuous self maps of $X$ defined by $D(f,g) = \sup_{x\in X}\overline{d}(f(x),g(x))$. We say that $f$ is a uniform equivalence, if both $f$ and $f^{-1}$ are uniformly continuous. The set of all uniformly continuous maps and the set of all uniform equivalences on $X$ are denoted by $UC(X)$ and $UE(X)$ respectively.   
\medskip

Let $f,f_n\in UC(X)$ for each $n\in\mathbb{N}^+$. Then, we recall that
\\
(i) $f_{n}$ is pointwise convergent to $f$ or $f_{n}\xrightarrow[]{\text{pc}} f$, if for each $x\in X$ and each $\epsilon > 0$ there exists an $N(x, \epsilon) = N\in \mathbb{N}^{+}$ such that $d(f_{n}(x), f(x)) < \epsilon$, for all $n\geq N$.
\\
(ii) $f_{n}$ is uniformly convergent to $f$ or $f_{n}\xrightarrow[]{\text{uc}} f$, if for every $\epsilon > 0$ there exists an $N(\epsilon) = N\in \mathbb{N}^{+}$ such that $d(f_{n}(x), f(x)) < \epsilon$, for all $n\geq N$ and for all $x\in X$.
\\
(iii) $f_{n}$ is orbitally convergent to $f$ or $f_{n}\xrightarrow[]{\text{oc}} f$, if for every $\epsilon > 0$ there exists an $N(\epsilon) = N\in \mathbb{N}^{+}$ such that $d(f_{n}^{k}(x), f^{k}(x)) < \epsilon$, for all $n\geq N$, for each $x\in X$ and for each $k\in \mathbb{N}^{+}$ \cite[Remark 5]{FDA}.
\medskip

A point $x\in X$ is said to be an expansive (positively expansive) point of $f\in UE(X)$ ($f\in UC(X)$) if there is a $\delta^{f}_{x}>0$ such that for every element $y\in X$ distinct from $x$, we have $d(f^n(x),f^n(y))>\delta^{f}_{x}$ for some $n\in\mathbb{Z}$ ($n\in \mathbb{N}$) \cite{RP}. The set of all expansive (positively expansive) points of $f$ is denoted by $E(f)$ ($E^{+}(f)$).  
\medskip
 
A map $f\in UE(X)$ is said to be expansive if for every pair of distinct points $x, y\in X$ there exists a constant $\delta > 0$ such that $d(f^{n}(x), f^{n}(y)) > \delta$, for some $n\in \mathbb{Z}$ \cite{UU}. We set $E_{x}(f, y, \epsilon) = \lbrace n\in \mathbb{Z} : d(f^{n}(x), f^{n}(y)) > \epsilon \rbrace$ and $E^{+}_{x}(f, y, \epsilon) = \lbrace n\in \mathbb{N} : d(f^{n}(x), f^{n}(y)) > \epsilon \rbrace$.
\medskip 

A point $x\in X$ is called an atom for a measure $\mu$ if $\mu(\lbrace x\rbrace)>0$. A measure $\mu$ on $X$ is said to be non-atomic if it has no atom. 
We call $X$ to be non-atomic if there exists a non-atomic Borel measure on it. 
Every Borel measure is assumed to be non-trivial i.e. $\mu(X) > 0$. Let $f\in UE(X)$ ($f\in UC(X)$). Then, a Borel measure $\mu$ on $X$ is said to be pointwise (positively pointwise) expansive for $f$ at $x\in X$, if there exists a $\delta_x>0$ such that $\mu(\Gamma^{f}_{\delta_x}(x))=0$ ($\mu(\Phi^{f}_{\delta_x}(x))=0$), where $\Gamma^{f}_{\delta_{x}}(x)=\lbrace y\in X\mid d(f^n(x),f^n(y))\leq \delta_{x}$, for each $n\in\mathbb{Z}\rbrace$ and $\Phi^{f}_{\delta_x}(x)=\lbrace y\in X\mid d(f^n(x),f^n(y))\leq \delta_x$, for each $n\in\mathbb{N}\rbrace$. The set of all points at which $\mu$ is pointwise (positively pointwise) expansive for $f$ is denoted by $E(f, \mu)$ ($E^{+}(f, \mu)$). If $X$ is non-atomic, then $f\in UE(X)$ ($f\in UC(X)$) is said to be pointwise (positively pointwise) measure expansive at $x\in X$ if there is a $\delta_x>0$ such that $\mu(\Gamma^{f}_{\delta_x}(x))=0$ ($\mu(\Phi^{f}_{\delta_x}(x))=0$) for any non-atomic Borel measure $\mu$. The set of all points at which $f$ is pointwise (positively pointwise) measure expansive is denoted by $EM(f)$ ($EM^{+}(f)$). A map $f\in UE(X)$ is said to be strongly pointwise measure expansive at $x\in X$ if there is a $\delta_x>0$ such that $\mu(\Gamma^{f}_{\delta_x}(x))=\mu(\lbrace x\rbrace)$ for any Borel measure $\mu$ on $X$. The set of all points at which $f$ is strongly pointwise (strongly positively pointwise) measure expansive is denoted by $ES(f)$ ($ES^{+}(f)$).
\medskip

A point $x\in X$ is said to be a sensitive point of $f\in UC(X)$ if there exists a $\delta^{f}_{x} > 0$ such that for every open set $U$ containing $x$, there exists a $y\in U$ satisfying $d(f^{n}(x), f^{n}(y))>\delta^{f}_{x}$, for some $n\in \mathbb{N}^{+}$. The set of all sensitive points of $f$ is denoted by $Se(f)$. 
We set $Se_{x}(f, \epsilon, \delta) = \lbrace (y, n)\in B(x, \epsilon)\times \mathbb{N}^{+} : d(f^{n}(x), f^{n}(y)) > \delta \rbrace$.
\medskip

A map $f\in UC(X)$ on $X$ is said to have a dense set of periodic points at $x\in X$, if every deleted open neighbourhood of $x$ contains a periodic point of $f$. The set of all such points with respect to $f$ is denoted by $P(f)$. We set $P(f, x, \epsilon) = \lbrace n\in \mathbb{N}^{+} : d(f^{n}(x), x) < \epsilon \rbrace$.
\medskip
 
For $f\in UC(X)$, any open set $U$ containing $x\in X$ and any non-empty open set $V$, we set $N(f, x, U, V) = \lbrace n\in \mathbb{N}^{+} : f^{n}(U)\cap V \neq \phi\rbrace$. A point $x\in X$ is said to be a topologically transitive point of $f$, if for any open set $U$ containing $x$ and any non-empty open set $V$, $N(f, x, U, V)\neq \phi$. A point $x\in X$ is said to be a topologically mixing point of $f$, if for any open set $U$ containing $x$ and any non-empty open set $V$, $N(f, x, U, V)$ is cofinite. The set of all topologically transitive (topologically mixing) points of $f$ is denoted by $Tt(f)$ ($Tm(f)$ respectively). We set $Tt(f, x, U, V) = \lbrace (y, n)\in U\times \mathbb{N}^{+} : f^{n}(y)\in V\rbrace$.
\medskip

A point $x\in X$ is said to be a Devaney chaotic point of $f\in UC(X)$, if $x\in Tt(f)\cap P(f)
\cap Se(f)$ \cite{DKDM}. The set of all Devaney chaotic points of $f$ is denoted by $Dc(f)$.

A sequence $\rho = \lbrace x_{i}\rbrace_{i\in \mathbb{Z}}$ ($\rho^{+} = \lbrace x_{i}\rbrace_{i\in \mathbb{N}}$) is said to be through a subset $B$ of $X$ if $x_{0}\in B$. We say that $\rho$ ($\rho^{+}$) is a $\delta_{x}^{f}$-pseudo orbit for $f\in UE(X)$ ($f\in UC(X)$) through $x$ if $d(f(x_{n}), x_{n+1}) < \delta_{x}^{f}$, for each $n\in \mathbb{Z}$ ($n\in \mathbb{N}$). We say that $\rho$ ($\rho^{+}$) is $\epsilon$-traced by $y\in X$ through $f$, if $d(f^{n}(y), x_{n}) < \epsilon$, for each $n\in \mathbb{Z}$ ($n\in \mathbb{N}$). A point $x\in X$ is said to be a shadowable point (positive shadowable point) for $f\in UE(X)$ ($UC(X)$) if for every $\epsilon > 0$, there exists a $\delta_{x}^{f}(\epsilon) > 0$ such that every $\delta_{x}^{f}(\epsilon)$-pseudo orbit $\rho$ ($\rho^{+}$) through $x$ can be $\epsilon$-traced. The set of all shadowable (positive shadowable) points of $f$ is denoted by $Sh(f)$ ($Sh^{+}(f)$). A map $f\in UE(X)$ ($UC(X)$) has shadowing (positive shadowing) if for every $\epsilon > 0$, there exists a $\delta^{f}(\epsilon) > 0$ such that every $\delta^{f}(\epsilon)$-pseudo orbit $\rho$ ($\rho^{+}$) through $X$ can be $\epsilon$-traced by a point in $X$ \cite{MSP}. We set $Sh(f, x, \epsilon) = \lbrace \delta > 0 :$ every $\delta$-pseudo orbit through $x$ for $f$ is $\epsilon$-traced through $f$ by a point in $X\rbrace$ and $Sh^{+}(f_{n}, x, \epsilon) = \lbrace \delta > 0 :$ every $\delta$-pseudo orbit $\lbrace x_{i}\rbrace_{i\in \mathbb{N}}$ through $x$ for $f$ is $\epsilon$-traced through $f$ by a point in $X\rbrace$. 
\medskip

A point $x\in X$ is said to be a specification point of $f\in UC(X)$ if for every $\epsilon > 0$, there exists a positive integer $M^{f}_{x}(\epsilon)$ such that for any finite sequence $x = x_{1}, x_{2}, ..., x_{k}$ in $X$ and any set of integers $0\leq a_{1}\leq b_{1} < a_{2}\leq b_{2} < . . .< a_{k}\leq b_{k}$ with $a_{j} - b_{j-1} \geq M^{f}_{x}(\epsilon)$, for all $1\leq j \leq k$, there exists a $y\in X$  such that $d(f^{i}(y), f^{i}(x_{j})) < \epsilon$, for all $a_{j}\leq i\leq b_{j}$ and for all $1\leq j\leq k$ \cite{DKDM}. The set of all specification points of $f$ is denoted by $Sp(f)$. A map $f$ has specification if choice of $M^{f}_{x}(\epsilon)$ depends only on $\epsilon$. We set $Sp(g, x, \epsilon) = \lbrace M\in \mathbb{N}^{+} : M$ corresponds to $\epsilon$ in the definition of specification point $x\rbrace$.  

\begin{theorem}\cite[Lemma 3.1]{AAU}
Let $(X, d)$ be a compact metric space, and suppose that the sequence of continuous functions $f_{n} : X\rightarrow X$, for each $n\in \mathbb{N}^{+}$, converges uniformly to $f: X\rightarrow X$. Then, for given $\epsilon > 0$ and a positive integer $k$ there exists a positive integer $n_{0}$ (possibly depending on $k$) such that for all $n > n_{0}$, $d(f_{n}^{k}(x), f^{k}(x)) < \epsilon$, for each $x \in X$.
\label{T2.2}
\end{theorem}

\begin{theorem}\cite[Theorem 4.8]{DKDM}
Let $f$ be a continuous map on an infinite metric space $X$. If $x\in X$ is a topologically transitive point such that $f$ has dense set of periodic points at $x$, then $x$ is a sensitive point of $f$.
\label{T2.1}
\end{theorem}

\begin{theorem}\cite[Lemma 4.1]{CLS}
If $f: [0, 1] \rightarrow [0, 1]$ is continuous and has fixed points only at the ends of the interval, then $f$ has the shadowing property.
\label{T2.3}
\end{theorem}

\begin{theorem}\cite[Lemma 3.11]{KLMP}
Every topologically stable point of a homeomorphism on a compact manifold of dimension atleast $2$, is a shadowable point.
\label{T2.6}
\end{theorem}

\begin{theorem}\cite[Lemma 3.14]{KLMP}
Every shadowable point of a homeomorphism on a compact metric space, is a $\alpha$-persistent point.
\label{T2.7}
\end{theorem}

\section{Sufficient conditions to be a Dynamic point of limit}
In this section, we aim to derive sufficient conditions under which the limit of a sequence of functions exhibits particular dynamical behaviour at a point like expansivity, $\mu$-expansivity, transitivity, mixing, Devaney chaos, shadowing, specification, topological stability etcetera. 
We need following notions to state and prove our main results.

Let $f,f_n\in UC(X)$ for each $n\in\mathbb{N}^+$. Then, 
\\
(i) $f_{n}$ is weak orbitally convergent to $f$ or $f_{n}\xrightarrow[]{\text{woc}} f$, if for every pair $k\in \mathbb{N}^{+}$ and $\epsilon > 0$, there exists an $N(\epsilon, k) = N\in \mathbb{N}^{+}$ such that $d(f_{n}^{k}(x), f^{k}(x)) < \epsilon$, for all $n\geq N$ and for each $x\in X$.   
\\
(ii) $f_{n}$ is pointwise weak orbitally convergent to $f$ or $f_{n}\xrightarrow[]{\text{pwoc}} f$, if for every triplet $x\in X$,  $k\in \mathbb{N}^{+}$ and $\epsilon>0$, there exists an $N(x, \epsilon, k) = N\in \mathbb{N}^{+}$ such that $d(f_{n}^{k}(x), f^{k}(x)) < \epsilon$, for all $n\geq N$.   
\medskip

\begin{Remark}
From corresponding definitions and Theorem \ref{T2.2}, we observe that every orbitally convergent sequence is uniformly convergent, every uniformly convergent sequence is weak orbitally convergent, every weak orbitally convergent sequence is pointwise weak orbitally convergent and every pointwise weak orbitally convergent sequence is pointwise convergent.  
\end{Remark}

Next examples shows that a uniformly convergent (and hence weak orbitally convergent) sequence need not be orbitally convergent.
\begin{Example}
Let $\alpha_{n}$ be a strictly increasing sequence of positive irrationals converges to $1$. Define $f_{n}$ by $f_{n}(x) = xe^{2\pi i \alpha_{n}}$, for each $x\in \mathbb{S}^{1}$ and for each $n\in \mathbb{N}^{+}$. Clearly, $f_{n}\xrightarrow[]{uc} I_{\mathbb{S}^{1}}$, where $I_{S^{1}}$ is the identity map on $\mathbb{S}^{1}$. Suppose that $f_{n}\xrightarrow[]{oc} I_{S^{1}}$. Then, for sufficiently small $\epsilon$, there exists an $N\in \mathbb{N}^{+}$ such that $d(f_{n}^{k}(x), x) < \epsilon$, for all $n\geq N$, for all $k\in \mathbb{N}$ and for each $x\in X$ implying $\lbrace f_{N}^{k}(x) : k\in \mathbb{N}\rbrace$ is not dense in $S^{1}$, which is a contradiction. Hence, $f_{n}$ is not  orbitally convergent to $I_{S^{1}}$.
\label{EN1}
\end{Example}

Notions of $\alpha$-persistent points and topologically stable points for homeomorphisms on compact metric spaces  \cite{KLMP} can be extended to homeomorphisms on arbitrary metric spaces. We provide these definitions below.
A point $x\in X$ is said to be an $\alpha$-persistent point of $f\in UE(X)$ if for every $\epsilon > 0$ there exists a $\delta_{x}^{f}(\epsilon) > 0$ such that for every $g\in UE(X)$ satisfying $D(f, g) < \delta_{x}^{f}(\epsilon)$, there is a $y\in X$ such that $d(f^{n}(y), g^{n}(x)) < \epsilon$, for each $n\in \mathbb{Z}$. The set of all $\alpha$-persistent points of $f$ is denoted by $P_{\alpha}(f)$. If for every $\epsilon > 0$ we can choose $\delta_{x}^{f}$ independent of choice of point $x\in X$, then we say that $f$ is $\alpha$-persistent. We set $P_{\alpha}(f, x, \epsilon) = \lbrace \delta > 0 : \delta$ corresponds to $\epsilon$ in the definition of $\alpha$-persistent point $x\rbrace$.
\medskip

A point $x\in X$ is said to be a topologically (weak topologically) stable point of $f\in UE(X)$ if for every $\epsilon > 0$, there exists a $\delta_{x}^{f}(\epsilon) > 0$ such that for every $g\in UE(X)$ satisfying $D(f, g) < \delta_{x}^{f}(\epsilon)$, there is a continuous map $h : \overline{\mathcal{O}}_{g}(x) \rightarrow X$ such that $f\circ h = h\circ g$ ($d(f^{n}(h(y)), g^{n}(y)) < \epsilon$, for each $n\in\mathbb{Z}$ and for each $y\in \overline{\mathcal{O}}_{g}(x)$) and $d(h(y), y) < \epsilon$, for each $y\in \overline{\mathcal{O}}_{g}(x)$, where $\mathcal{O}_{g}(x) = \lbrace g^{n}(x) : n\in \mathbb{Z}\rbrace$. The set of all topologically (weak topologically) stable points of $f$ is denoted by $Ts(f)$ ($Wts(f)$). We set $Wts(f, x, \epsilon) = \lbrace \delta > 0 : \delta$ corresponds to $\epsilon$ in the definition of weak topologically stable point $x\rbrace$.  
We say that $f\in UE(X)$ is topologically stable (weak topologically stable) if for every $\epsilon > 0$, there exists a $\delta > 0$ such that for every $g\in UE(X)$ satisfying $D(f, g) < \delta$, there is a continuous map $h : X \rightarrow X$ such that $f\circ h = h\circ g$ ($d(f^{n}(h(x)), g^{n}(x)) < \epsilon$, for each $n\in\mathbb{Z}$ and for each $x\in X$) and $d(h(x), x) < \epsilon$, for each $x\in X$.
\medskip

\begin{theorem}
Let $\lbrace f_{n}\rbrace_{n\in \mathbb{N}^{+}}$ be a sequence of uniformly continuous maps on $X$. If $\lbrace f_{n}\rbrace_{n\in \mathbb{N}^{+}}$ is pointwise weak orbitally converging to $f\in UC(X)$, then the following statements are true:
\begin{enumerate}
\item[(i)] $x\in E^{+}(f)$ if and only if there exists a $ \delta > 0$ such that  $\cup_{m\geq 1} \cap_{n\geq m} E^{+}_{x}(f_{n}, y, \delta)\neq \phi$, for each $y\in X$.
\item[(ii)] $x\in Se(f)$ if and only if there exists a $\delta > 0$ such that for every $\epsilon > 0$, $\cup_{m\geq 1} \cap_{n\geq m} Se_{x}(f_{n}, \epsilon, \delta)\neq \phi$.
\item[(iii)] $x\in P(f)$ if and only if for every $\epsilon > 0$, $\cup_{z\in B^{-}(x, \epsilon)} \cap_{\delta > 0}\cup_{m\geq 1}\cap_{n\geq m} P(f_{n}, z, \delta)\neq \phi$.
\item[(iv)] $x\in Tt(f)$ if and only if for every non-empty open set $U$ containing $x$ and every non-empty open set $V$, $\cup_{m\geq 1}\cap_{n\geq m} Tt(f_{n}, x, U, V)\neq \phi$.
\item[(v)] $x\in Tm(f)$ if and only if for every non-empty open set $U$ containing $x$ and every non-empty open set $V$, the image of the projection map $P: \cup_{m\geq 1}\cap_{n\geq m} Tt(f_{n}, x, U, V)\rightarrow N^{+}$, is cofinite.
\item[(vi)] If $f_{n}, f\in UE(X)$ for each $n\in \mathbb{N}^{+}$ and $f_{n}^{-1}\xrightarrow{pwoc} f^{-1}$, then $x\in E(f)$ if and only if there exists a $\delta > 0$ such that  $\cup_{m\geq 1} \cap_{n\geq m} E_{x}(f_{n}, y, \delta)\neq \phi$, for each $y\in X$.
\end{enumerate}
\label{T4.1}
\end{theorem}
\begin{proof}
Suppose that $\lbrace f_{n}\rbrace_{n\in \mathbb{N}^{+}}$ is a sequence of uniformly continuous maps on $X$ pointwise weak orbitally converging to $f\in UC(X)$:
\begin{enumerate}
\item[(i)] Let $x\in E^{+}(f)$ with pointwise expansivity constant $\delta_{x}^{f} > 0$ and choose $y\in X$. Fix $k(y) = k \in \mathbb{N}$ such that $d(f^{k}(x), f^{k}(y)) > \delta_{x}^{f}$. Since $f_{n}\xrightarrow[]{pwoc} f$, we have $f_{n}^{k}\xrightarrow[]{pc} f^{k}$. Choose $N(y) = N\in \mathbb{N}^{+}$ such that $d(f_{n}^{k}(z), f^{k}(z)) < \frac{\delta_{x}^{f}}{3}$, for all $z\in \lbrace x, y\rbrace$ and for all $n\geq N$. Note that, $d(f_{n}^{k}(x), f_{n}^{k}(y)) > \frac{\delta_{x}^{f}}{3}$, for all $n\geq N$ implying $k\in \cap_{n\geq N} E^{+}_{x}(f_{n}, y, \frac{\delta_{x}^{f}}{3})$. Since $y$ is chosen arbitrarily, we have $\cup_{m\geq 1} \cap_{n\geq m} E^{+}_{x}(f_{n}, y, \frac{\delta_{x}^{f}}{3})\neq \phi$, for each $y\in X$.
\medskip

Conversely, choose $\delta > 0$ such that  $\cup_{m\geq 1} \cap_{n\geq m} E^{+}_{x}(f_{n}, y, \delta)\neq \phi$, for each $y\in X$. Let $y\in X$, $M(y) = M\in \mathbb{N}^{+}$ such that $\cap_{n\geq M} E^{+}_{x}(f_{n}, y, \delta)\neq \phi$ and $p\in \cap_{n\geq M} E^{+}_{x}(f_{n}, y, \delta)$. Since $f_{n}\xrightarrow[]{pwoc} f$, we have $f_{n}^{p}\xrightarrow[]{pc} f^{p}$. Choose $N(p)=N\geq M$ such that $d(f_{n}^{p}(z), f^{p}(z)) < \frac{\delta}{3}$, for all $z\in \lbrace x, y\rbrace$ and for all $n\geq N$. Then we must have $d(f^{p}(x), f^{p}(y)) > \frac{\delta}{3}$. Since $\delta$ does not depends on $y$, we get that $x\in E^{+}(f)$ with pointwise expansivity constant $\frac{\delta}{3}$.
\medskip

\item[(ii)] Proof is similar to the proof of part (i).
\medskip

\item[(iii)] Let $x\in P(f)$. For given $\epsilon > 0$, assume that $y\in B^{-}(x, \epsilon)$ is a periodic point of period $k\in \mathbb{N}^{+}$. Since $f_{n}\xrightarrow[]{pwoc} f$, we have $f_{n}^{k}\xrightarrow[]{pc} f^{k}$. Therefore for given $\delta > 0$, we can choose $N\in \mathbb{N}^{+}$ such that $d(f_{n}^{k}(y), y) < \delta$, for all $n\geq N$. 
Hence $k\in P(f_{n}, y, \delta)$ for all $n\geq N$ implying $k\in \cap_{n\geq N} P(f_{n}, y, \delta)$ i.e. $k\in \cup_{m\geq 1}\cap_{n\geq m} P(f_{n}, y, \delta)$. 
Since $\delta$ is chosen arbitrarily, $k\in \cup_{m\geq 1}\cap_{n\geq m} P(f_{n}, y, \delta)$ for each $\delta > 0$ implying $k\in$ 
$\cap_{\delta > 0} \cup_{m\geq 1}\cap_{n\geq m} P(f_{n}, y, \delta)$ and hence $\cup_{z\in B^{-}(x, \epsilon)} \cap_{\delta > 0}\cup_{m\geq 1}\cap_{n\geq m} P(f_{n}, z, \delta)\neq \phi$. 
\medskip

Conversely, for given $\epsilon > 0$ assume that $\cup_{z\in B^{-}(x, \epsilon)} \cap_{\delta > 0}\cup_{m\geq 1}\cap_{n\geq m} P(f_{n}, z, \delta)\neq \phi$. Choose $y\in B^{-}(x, \epsilon)$ and $k\in \mathbb{N}^{+}$ such that $k\in \cup_{m\geq 1}\cap_{n\geq m} P(f_{n}, y, \delta)$, for each $\delta > 0$. Thus every $\delta$-ball centred at $y$ contains a tail of sequence $\lbrace f_{n}^{k}(y)\rbrace_{n\in \mathbb{N}^{+}}$ implying $f_{n}^{k}(y)$ converges to $y$. Since $f_{n}^{k}\xrightarrow{pc} f^{k}$, we have $f^{k}(y) = y$. Since $\epsilon$ is chosen arbitrarily, $x\in P(f)$.
\medskip

\item[(iv)] Let $x\in Tt(f)$. Then, for given non-empty open set $U$ containing $x$ and non-empty open set $V$, there exists a $k\in  \mathbb{N}^{+}$ satisfying $f^{k}(U)\cap V \neq \phi$. Thus we can choose $y\in U$ such that $f^{k}(y)\in V$. Since $f_{n}^{k}\xrightarrow{pc} f^{k}$, there exists an $N\in \mathbb{N}^{+}$  such that $f_{n}^{k}(y) \in V$, for all $n\geq N$. Hence $(y, k)\in \cap_{n\geq N} Tt(f_{n}, x, U, V)$ implying $\cup_{m\geq 1}\cap_{n\geq m} Tt(f_{n}, x, U, V)\neq \phi$. 
\medskip

Conversely, choose non-empty open set $U$ containing $x$ and  non-empty open set $V$. Let $y\in U$, $z\in V$ and $\epsilon > 0$ such that $B(y, \epsilon) \subset U$ and $B(z, \epsilon)\subset V$. Set $U' = B(y, \frac{\epsilon}{4})$ and $V' = B(z, \frac{\epsilon}{4})$. By assumption, choose $w\in U'$ and $k, N\in \mathbb{N}^{+}$ such that $(w, k)\in \cap_{n\geq N} Tt(f_{n}, x, U', V')$. Since $f_{n}^{k}\xrightarrow{pc} f^{k}$, choose $M\in \mathbb{N}^{+}$ such that $d(f_{n}^{k}(w), f^{k}(w)) < \frac{\epsilon}{4}$, for all $n\geq M$. Thus for all $n\geq max\lbrace N, M\rbrace$, $d(f_{n}^{k}(w), z)< \frac{\epsilon}{4}$ and $d(f_{n}^{k}(w), f^{k}(w)) < \frac{\epsilon}{4}$. Hence $f^{k}(w)\in B(z, \frac{\epsilon}{2})\subset V$ implying $f^{k}(U)\cap V\neq \phi$. Since $U$ and $V$ are chosen arbitrarily, we get that $x\in Tt(f)$.
\medskip

\item[(v)] Let $x\in Tm(f)$. Then for given non-empty open set $U$ containing $x$ and non-empty open set $V$, there exists a $k\in  \mathbb{N}^{+}$ satisfying $f^{r}(U)\cap V \neq \phi$, for all $r\geq k$. Hence for each $r\geq k$, we can choose $y_{r}\in U$ such that $f^{r}(y_{r})\in V$. Since $f_{n}^{r}\xrightarrow{pc} f^{r}$ for each $r\geq k$, there exists an $N_{r}\in \mathbb{N}^{+}$  such that $f_{n}^{r}(y_{r}) \in V$, for all $n\geq N_{r}$ and for each $r\geq k$. Thus $(y_{r}, r)\in \cap_{n\geq N_{r}} Tt(f_{n}, x, U, V)$ for each $r\geq k$ implying $(y_{r}, r)\in \cup_{m\geq 1}\cap_{n\geq m} Tt(f_{n}, x, U, V)\neq \phi$, for each $r\geq k$ i.e. the image of $P$ contains $\lbrace r\in \mathbb{N}^{+} : r\geq k \rbrace$. Hence the image of the projection map $P$ is cofinite. 
\medskip

Conversely, choose non-empty open set $U$ containing $x$ and non-empty open set $V$. Let $y\in U$, $z\in V$ and $\epsilon > 0$ such that $B(y, \epsilon) \subset U$ and $B(z, \epsilon)\subset V$. Set $U' = B(y, \frac{\epsilon}{4})$ and $V' = B(z, \frac{\epsilon}{4})$. Suppose that the image of the projection map $P: \cup_{m\geq 1}\cap_{n\geq m} Tt(f_{n}, x, U', V')\rightarrow N^{+}$ is cofinite. Thus there exists a $k\in \mathbb{N}^{+}$ such that $\lbrace r\in \mathbb{N}^{+} : r\geq k\rbrace$ is contained in the image of $P$. For each $r\geq k$, choose $y_{r}\in U'$ such that $(y_{r}, r)\in \cup_{m\geq 1}\cap_{n\geq m} Tt(f_{n}, x, U', V')$. For each $r\geq k$, fix $N_{r}\in \mathbb{N}^{+}$ such that $(y_{r}, r)\in \cap_{n\geq N_{r}} Tt(f_{n}, x, U', V')$. Since $f_{n}^{r}\xrightarrow[]{pc} f^{r}$, choose $M_{r}\in \mathbb{N}^{+}$ such that $d(f_{n}^{r}(y_{r}), f^{r}(y_{r})) < \frac{\epsilon}{4}$, for all $n\geq M_{r}$ and for each $r\geq k$. Set $Q_{r} = max \lbrace N_{r}, M_{r}\rbrace$. Thus for each $r\geq k$, $d(f_{n}^{r}(y_{r}), z)< \frac{\epsilon}{4}$ and $d(f_{n}^{r}(y_{r}), f^{r}(y_{r})) < \frac{\epsilon}{4}$, for all $n\geq Q_{r}$. Hence $f^{r}(y_{r})\in B(z, \frac{\epsilon}{2})\subset V$, for each $r\geq k$ implying $f^{r}(U)\cap V\neq \phi$ for all $r\geq k$. Hence $x\in Tm(f)$.
\medskip

\item[(vi)] Proof is similar to the proof of part (i).
\end{enumerate}
\end{proof}

\begin{Corollary}
Let $\lbrace f_{n}\rbrace_{n\in \mathbb{N}^{+}}$ be a sequence of uniformly continuous maps on an infinite metric space $X$. If $\lbrace f_{n}\rbrace_{n\in \mathbb{N}^{+}}$ is pointwise weak orbitally converging to $f\in UC(X)$, then $x\in Dc(f)$ if and only if the following conditions holds:
\begin{enumerate}
\item For every $\epsilon > 0$, $\cup_{z\in B^{-}(x, \epsilon)} \cap_{\delta > 0}\cup_{m\geq 1}\cap_{n\geq m} P(f_{n}, z, \delta)\neq \phi$.
\item For every non-empty open set $U$ containing $x$ and every non-empty open set $V$, $\cup_{m\geq 1}\cap_{n\geq m} Tt(f_{n}, x, U, V)\neq \phi$.
\end{enumerate}
\label{C4.2}
\end{Corollary}
\begin{proof}
Proof follows from Theorem \ref{T2.1}, Theorem \ref{T4.1}(iii) and Theorem \ref{T4.1}(iv).
\end{proof}

\begin{theorem}
Let $\lbrace f_{n}\rbrace_{n\in \mathbb{N}^{+}}$ be a sequence of uniformly continuous maps on $X$. If $\lbrace f_{n}\rbrace_{n\in \mathbb{N}^{+}}$ is pointwise weak orbitally converging to map $f\in UC(X)$ and $\mu$ is a non-atomic Borel measure on $X$, then the following statements are true:
\begin{enumerate}
\item[(i)] $x\in E^{+}(f, \mu)$ if and only if there exists a $ \delta > 0$ such that $\mu(\lbrace y\in X : \cup_{m\geq 1} \cap_{n\geq m} E^{+}_{x}(f_{n}, y, \delta) = \phi\rbrace) = 0$.
\item[(ii)] $x\in EM^{+}(f)$ if and only if there exists a $\delta > 0$ such that $\nu(\lbrace y\in X : \cup_{m\geq 1} \cap_{n\geq m} E^{+}_{x}(f_{n}, y, \delta) = \phi \rbrace) = 0$, for every non-atomic Borel measure $\nu$ on $X$.
\item[(iii)] $x\in ES^{+}(f)$ if and only if there exists a $\delta > 0$ such that $\nu(\lbrace y\in X : \cup_{m\geq 1} \cap_{n\geq m} E^{+}_{x}(f_{n}, y, \delta) = \phi \rbrace) = \mu(x)$, for every Borel measure $\nu$ on $X$.
\item[(iv)] If $f_{n},f\in UE(X)$, for each $n\in \mathbb{N}^{+}$ and $f_{n}^{-1}\xrightarrow{pwoc} f^{-1}$, then \\
(a) $x\in E(f, \mu)$ if and only if there exists a $\delta > 0$ such that $\mu(\lbrace y\in X : \cup_{m\geq 1} \cap_{n\geq m} E_{x}(f_{n}, y, \delta) = \phi\rbrace) = 0$.\\
(b) $x\in EM(f)$ if and only if there exists a $\delta > 0$ such that $\nu(\lbrace y\in X : \cup_{m\geq 1} \cap_{n\geq m} E_{x}(f_{n}, y, \delta) = \phi \rbrace) = 0$, for every non-atomic Borel measure $\nu$ on $X$.\\
(c) $x\in ES(f)$ if and only if there exists a $\delta > 0$ such that $\nu(\lbrace y\in X : \cup_{m\geq 1} \cap_{n\geq m} E_{x}(f_{n}, y, \delta) = \phi \rbrace) = \mu(x)$, for every Borel measure $\nu$ on $X$.
\end{enumerate}
\label{T4.3}
\end{theorem}
\begin{proof}
Following the steps as in the proof of Theorem \ref{T4.1}(i), it is easy to check that for every $\delta > 0$ and fixed $x\in X$ we have $\Phi^{f}_{\frac{\delta}{3}}(x)\subset \lbrace y\in X : \cup_{m\geq 1} \cap_{n\geq m} E^{+}_{x}(f_{n}, y, \delta) = \phi\rbrace$, $\lbrace y\in X : \cup_{m\geq 1} \cap_{n\geq m} E^{+}_{x}(f_{n}, y, \frac{\delta}{3}) = \phi\rbrace \subset \Phi^{f}_{\delta}(x)$, $\Gamma^{f}_{\frac{\delta}{3}}(x)\subset \lbrace y\in X : \cup_{m\geq 1} \cap_{n\geq m} E_{x}(f_{n}, y, \delta) = \phi\rbrace$ and $\lbrace y\in X : \cup_{m\geq 1} \cap_{n\geq m} E_{x}(f_{n}, y, \frac{\delta}{3}) = \phi\rbrace \subset \Gamma^{f}_{\delta}(x)$. Now, proofs of (i), (ii), (iii) and (iv) follow using above relations.
\end{proof}

\begin{theorem}
Let $\lbrace f_{n}\rbrace_{n\in \mathbb{N}^{+}}$ be a sequence of uniformly continuous maps on $X$. If $\lbrace f_{n}\rbrace_{n\in \mathbb{N}^{+}}$ is orbitally converging to $f\in UC(X)$, then the following statements are true:
\begin{enumerate}
\item[(i)] $x\in Sh^{+}(f)$ if and only if for every $\epsilon > 0$, $\cup_{m\geq 1}\cap_{n\geq m}Sh^{+}(f_{n}, x, \epsilon)\neq \phi$.
\item[(ii)] $x\in Sp(f)$ if and only if for every $\epsilon > 0$, $\cup_{m\geq 1}\cap_{n\geq m}Sp(f_{n}, x, \epsilon)\neq \phi$.
\end{enumerate}
\label{T4.4}
\end{theorem}
\begin{proof}
Suppose that $\lbrace f_{n}\rbrace_{n\in \mathbb{N}^{+}}$ is a sequence of uniformly continuous maps on $X$ orbitally converging to $f\in UC(X)$.
\begin{enumerate}
\item[(i)] Let $x\in Sh^{+}(f)$. For given $\epsilon > 0$, choose $0 < \delta < \epsilon$ such that every $\delta$-pseudo orbit for $f$ can be $\frac{\epsilon}{2}$-traced through $f$ by some point of $X$. Choose $N\in \mathbb{N}^{+}$ such that $d(f_{n}^{m}(y), f^{m}(y)) < \frac{\delta}{2}$, for all $n\geq N$, for each $y\in X$ and for each $m \in \mathbb{N}^{+}$. Let us fix $k\geq N$ and let $\gamma = \lbrace x_{i}\rbrace_{i\in \mathbb{N}}$ be a $\frac{\delta}{2}$-pseudo orbit for $f_{k}$ through $x$. Since $d(f(x_{i}), x_{i+1}) \leq d(f(x_{i}), f_{k}(x_{i})) + d(f_{k}(x_{i}), x_{i+1}) < \delta$ for all $i\in \mathbb{N}$, we get that $\gamma$ is a $\delta$-pseudo orbit for $f$. Hence there exists a $z\in X$ which $\frac{\epsilon}{2}$-traces $\gamma$ through $f$. Therefore, $d(f_{k}^{i}(z), x_{i}) \leq d(f_{k}^{i}(z), f^{i}(z)) + d(f^{i}(z), x_{i}) < \frac{\epsilon}{2} + \frac{\epsilon}{2} =\epsilon$, for each $i\in \mathbb{N}$. Hence $\gamma$ can be $\epsilon$-traced through $f_{k}$. Since $k$ and $\gamma$ are chosen arbitrary, for all $n\geq N$, every $\frac{\delta}{2}$-pseudo orbit through $x$ for $f_{n}$ can be $\epsilon$-traced through $f_{n}$ implying $\frac{\delta}{2}\in \cap_{n\geq N}Sh^{+}(f_{n}, x, \epsilon)$ i.e. $\cup_{m\geq 1}\cap_{n\geq m}Sh^{+}(f_{n}, x, \epsilon)\neq \phi$. Since $\epsilon$ is chosen arbitrarily, $\cup_{m\geq 1}\cap_{n\geq m}Sh^{+}(f_{n}, x, \epsilon)\neq \phi$.
\medskip

Conversely, for given $\epsilon > 0$, suppose that $\cup_{m\geq 1}\cap_{n\geq m}Sh^{+}(f_{n}, x, \frac{\epsilon}{2})\neq \phi$. Choose $N\in \mathbb{N}^{+}$ and $\delta > 0$ such that for each $n\geq N$, every $\delta$-pseudo orbit through $x$ for $f_{n}$ can be $\frac{\epsilon}{2}$-traced through $f_{n}$ by some point of $X$.
Choose $M\geq N$ such that $d(f_{n}^{m}(y), f^{m}(y)) < \frac{\delta}{2}$, for all $n\geq M$, for each $y\in X$ and for each $m\in \mathbb{N}^{+}$. Let $\gamma = \lbrace x_{i}\rbrace_{i\in \mathbb{N}}$ be a $\frac{\delta}{2}$-pseudo orbit for $f$ through $x$. Clearly, $\gamma$ is a $\delta$-pseudo orbit for $f_{M}$ through $x$ and hence can be $\frac{\epsilon}{2}$-traced by some $z\in X$ through $f_{M}$. It is easy to check that, $\gamma$ can be $\epsilon$-traced by $z\in X$ through $f$. Hence $x\in Sh^{+}(f)$.
\medskip

\item[(ii)] Let $x\in Sp(f)$. For given $\epsilon > 0$, choose $M\in \mathbb{N}^{+}$ corresponding to $\frac{\epsilon}{3}$ as in the definition of specification point. Choose $N\in \mathbb{N}^{+}$ such that $d(f_{n}^{m}(y), f^{m}(y)) < \frac{\epsilon}{3}$, for all $n\geq N$, for each $y\in X$ and for each $m \in \mathbb{N}^{+}$. Choose a finite sequence $x = x_{1}, x_{2}, ..., x_{k}$ in $X$ and any set of integers $0\leq a_{1}\leq b_{1} < a_{2}\leq b_{2} < . . .< a_{k}\leq b_{k}$ with $a_{j} - b_{j-1} \geq M$, for all $1\leq j \leq k$. Choose $y\in X$ such that $d(f^{i}(y), f^{i}(x_{j})) < \frac{\epsilon}{3}$, for all $a_{j}\leq i\leq b_{j}$ and for all $1\leq j\leq k$. Note that, $d(f_{n}^{i}(y), f_{n}^{i}(x_{j})) < \epsilon$, for all $a_{j}\leq i\leq b_{j}$, for all $1\leq j\leq k$ and for all $n\geq N$. Hence $M\in \cap_{n\geq N}Sp(f_{n}, x, \epsilon)$ implying
$\cup_{m\geq 1}\cap_{n\geq m}Sp(f_{n}, x,  \epsilon)\neq \phi$. Since $\epsilon$ is chosen arbitrarily, we get the conclusion.

Conversely, for given $\epsilon > 0$, suppose that $\cup_{m\geq 1}\cap_{n\geq m}Sp(f_{n}, x, \frac{\epsilon}{3})\neq \phi$. Choose $N\in \mathbb{N}^{+}$ such that $d(f_{n}^{m}(y), f^{m}(y)) < \frac{\epsilon}{3}$, for all $n\geq N$, for each $y\in X$ and for each $m \in \mathbb{N}^{+}$. Choose $K, M\in \mathbb{N}^{+}$ such that for all $n\geq K$, any finite sequence $x = x_{1}, x_{2}, ..., x_{k}$ in $X$ and any set of integers $0\leq a_{1}\leq b_{1} < a_{2}\leq b_{2} < . . .< a_{k}\leq b_{k}$ with $a_{j} - b_{j-1} \geq M$, for all $1\leq j \leq k$, there exists $y_{n}\in X$ such that $d(f_{n}^{i}(y_{n}), f_{n}^{i}(x_{j})) < \frac{\epsilon}{3}$, for all $a_{j}\leq i\leq b_{j}$, for all $1\leq j\leq k$ and for all $n\geq K$. Set $Q = max\lbrace N, K\rbrace$. Note that, $d(f^{i}(y_{Q}), f^{i}(x_{j})) < \epsilon$, for all $a_{j}\leq i\leq b_{j}$ and for all $1\leq j\leq k$. Since $\epsilon$ is chosen arbitrarily, $x\in Sp(f)$.
\end{enumerate}
\end{proof}

\begin{Lemma}
Let $f\in UE(X)$ be expansive on a Mandelkern locally compact space $X$. If $x\in X$, then the following statements are equivalent:
\begin{enumerate}
\item[(i)] $x$ is a topologically stable point.
\item[(ii)] $x$ is a weak topologically stable point.
\item[(iii)] $x$ is an $\alpha$-persistent point.
\end{enumerate}
\label{T3.12}
\end{Lemma}
\begin{proof}
Let $f$ be expansive with expansivity constant $\mathfrak{c}$. We claim that, for each $x\in X$ and for each $\epsilon > 0$, there exists an $N = N(x, \epsilon)\in \mathbb{N}$ such that $\sup_{|n|\leq N}d(f^{n}(x), f^{n}(y))$ $\leq \mathfrak{c}$ implies that $d(x, y)< \epsilon$. Otherwise, choose a pair $x\in X$ and $\epsilon > 0$ such that for each $N\in \mathbb{N}$ there exists an $x_{N}\in X$ satisfying $\sup_{|n|\leq N}d(f^{n}(x), f^{n}(x_{N}))\leq \mathfrak{c}$ and $d(x, x_{N}) \geq \epsilon$. Since $B[x, \mathfrak{c}]$ is compact, we can assume that $x_{N}\rightarrow x'$, for some $x'\in X$. Therefore, $d(f^{n}(x), f^{n}(x'))\leq \mathfrak{c}$, for each $n\in \mathbb{Z}$ and $d(x, x')\geq \epsilon$, which contradicts the expansivity of $f$ and hence the claim follows. 
Now, proof of (i)$\Rightarrow$(ii) and (ii)$\Rightarrow$(iii) follows from corresponding definitions. Recall that, if $Y$ and $Z$ are metric spaces such that $Z$ is complete, $S$ is dense in $Y$ and $k: S\rightarrow Z$ is continuous, then $k$ can be extended to a continuous function $K: Y\rightarrow Z$. 
Since every Mandelkern locally compact space is complete, one can use the above claim and follow similar steps as in the proof of \cite[Lemma 3.15]{KLMP} to show that (iii) $\Rightarrow$ (i).
\end{proof}

\begin{theorem}
Let $\lbrace f_{n}\rbrace_{n\in \mathbb{N}^{+}}$ be a sequence of uniform equivalences in $X$ such that $f_{n}\xrightarrow{oc} f$ and $f_{n}^{-1}\xrightarrow{oc} f^{-1}$, where $f\in UE(X)$. Then, the following statements are true:
\begin{enumerate}
\item[(i)] $x\in Sh(f)$ if and only if for every $\epsilon > 0$, $\cup_{m\geq 1}\cap_{n\geq m}Sh(f_{n}, x, \epsilon)\neq \phi$
\item[(ii)] $x\in P_{\alpha}(f)$ if and only if for every $\epsilon > 0$, $\cup_{m\geq 1}\cap_{n\geq m} P_{\alpha}(f_{n}, x, \epsilon)\neq \phi$. 
\item[(iii)] $x\in Wts(f)$ if and only if for every $\epsilon > 0$, $\cup_{m\geq 1}\cap_{n\geq m} Wts(f_{n}, x, \epsilon)\neq \phi$. 
\item[(iv)] If $f$ is expansive and $X$ is Mandelkern locally compact, then $x\in Ts(f)$ if and only if for every $\epsilon > 0$, $\cup_{m\geq 1}\cap_{n\geq m} Wts(f_{n}, x, \epsilon)\neq \phi$.
\end{enumerate}
\label{T4.5}
\end{theorem}
\begin{proof}
Suppose that $\lbrace f_{n}\rbrace_{n\in \mathbb{N}^{+}}$ is a sequence of uniform equivalences in $X$ such that $f_{n}\xrightarrow{oc} f$ and $f_{n}^{-1}\xrightarrow{oc} f^{-1}$, where $f$ is a uniform equivalence on $X$.
\begin{enumerate}
\item[(i)] Proof is similar to the proof of Theorem \ref{T4.4}(i).
\medskip

\item[(ii)] Let $x\in P_{\alpha}(f)$ and fix $0 < \epsilon < 1$. Choose $0 < \delta < \epsilon$, such that $\delta$ corresponds to $\frac{\epsilon}{2}$ as in the definition of $\alpha$-persistent points. Since $f_{n}\xrightarrow{oc} f$ and $f_{n}^{-1}\xrightarrow{oc} f^{-1}$, choose $N\in \mathbb{N}^{+}$ such that $d(f^{i}_{n}(y), f^{i}(y)) < \frac{\delta}{2}$ and $d(f^{-i}_{n}(y), f^{-i}(y)) < \frac{\delta}{2}$, for all $n\geq N$, for each $y\in X$ and for each $i\in \mathbb{N}^{+}$. Fix $k \geq N$ and choose a uniform equivalence $g_{k}$, such that $D(f_{k}, g_{k}) < \frac{\delta}{2}$. Clearly, $D(f, g_{k}) < \delta$ and hence there is a $y\in X$ such that $d(f^{n}(y), g_{k}^{n}(x)) < \frac{\epsilon}{2}$, for each $n\in \mathbb{Z}$. Note that, $d(f_{k}^{n}(y), g_{k}^{n}(x)) < \epsilon$, for each $n\in \mathbb{Z}$. Since $k$ is chosen arbitrarily, $\cap_{n\geq N} P_{\alpha}(f_{n}, x, \epsilon)\neq \phi$ implying $\cup_{m\geq 1}\cap_{n\geq m} P_{\alpha}(f_{n}, x, \epsilon)\neq \phi$. Since $\epsilon$ is chosen arbitrary, we get the conclusion.
\medskip

Conversely, fix $0 < \epsilon < 1$. Choose $N\in \mathbb{N}^{+}$ and $0 < \delta < \epsilon$ such that $\delta\in \cap_{n\geq N} P_{\alpha}(f_{n}, x, \frac{\epsilon}{2})$. Since $f_{n}\xrightarrow{oc} f$ and $f_{n}^{-1}\xrightarrow{oc} f^{-1}$, fix $M\geq N$ such that $d(f^{i}_{n}(y), f^{i}(y)) < \frac{\delta}{2}$ and $d(f^{-i}_{n}(y), f^{-i}(y)) < \frac{\delta}{2}$, for all $n\geq M$, for each $y\in X$ and for each $i\in \mathbb{N}^{+}$. Choose a uniform equivalence $g$, such that $D(f, g) < \frac{\delta}{2}$. Clearly $D(f_{M}, g) < \delta$ and hence there is a $y\in X$ such that $d(f_{M}^{n}(y), g^{n}(x)) < \frac{\epsilon}{2}$, for each $n\in \mathbb{Z}$. Note that, $d(f^{n}(y), g^{n}(x)) < \epsilon$, for each $n\in \mathbb{Z}$. Since $\epsilon$ is chosen arbitrary, $x\in P_{\alpha}(x)$.
\medskip

\item[(iii)] Proof is similar to the proof of (ii).
\medskip

\item[(iv)] Using Lemma \ref{T3.12} and (iii), we get the result.
\end{enumerate}
\end{proof}

\section{Examples}
We now discuss the topological nature of the set of all points with particular dynamical property like expansivity, sensitivity, transitivity etc. 
Particularly, the next example shows that the set of all expansive, positively expansive and sensitive points are neither open nor closed in general and the set of all points with dense set of periodic points in its neighbourhood, topologically transitive and topologically mixing points need not be an open set in general.
\begin{Example}
Let $X = [0, 1]$ be equipped with the Euclidean metric. For each $k\in \mathbb{N}^{+}$, $P_{k}$ denotes the mid-point of $[\frac{1}{k+1},\frac{1}{k}]$, $Q_{k}$ denotes the mid point of $[\frac{1}{k+1},P_{k}]$ and $R_{k}$ denotes the mid point of $[P_{k},\frac{1}{k}]$. Consider piecewise linear maps $f,g\in UE(X)$ defined as follows: 
\begin{center}
$f(0) = 0$, $f(\frac{1}{n}) = \frac{1}{n}$, $f(P_{n}) = Q_{n}$, for each $n\in\mathbb{N}^{+}$\\ 
\vspace*{0.2cm}
$g(0) = 0$, $g(\frac{1}{n}) = \frac{1}{n}$, $g(P_{2n}) = R_{2n}$, $g(P_{2n-1}) = Q_{2n-1}$ for each $n\in\mathbb{N}^{+}$
\end{center}
\vspace*{0.1cm}
Define a homeomorphism $h$ on $X$ by $h(x) = x^{2}$ for each $x\in X$. Then,
\begin{enumerate}
\item $E^{+}(g) = \lbrace \frac{1}{2n-1} : n\in \mathbb{N}^{+}\rbrace$ and $E(f) = \lbrace \frac{1}{n} : n\in \mathbb{N}^{+}\rbrace = Se(f)$. 
\item Since $0\notin E(f), Se(f), E^{+}(g)$, set of all expansive points, positively expansive points and sensitive points are neither closed nor open in general.
\item Since $P(f) = \lbrace 0 \rbrace$, $P(f)$ need not be an open set in general. 
\item Since $0\in P(f)\setminus Tt(f)$, Theorem \ref{T2.1} is not true if point is not transitive.
\item Since $Tt(h) = \lbrace 1\rbrace = Tm(h)$, set of all topologically transitive points and topologically mixing points need not be open in general.
\end{enumerate}
\label{E3.1}
\end{Example}

\begin{Remark}
For any map $f\in UC(X)$, $P(f)$, $Tt(f)$ and $Tm(f)$ are closed sets.
\label{E3.2}
\end{Remark}

Through Example \ref{E3.3} - Example \ref{E3.10}, we show that the expansive, positively expansive, sensitive, denseness of periodic points in its neighbourhood, topologically transitive, topologically mixing, shadowing and positive shadowing nature of a point under a sequence of function can not be transfer to its uniform limit. Through Example \ref{E3.13}, we show that the topological stability, weak topological stability and $\alpha$-persistence nature of a point under a sequence of functions can not be transfer to its pointwise limit.
 
\begin{Example}
Let $X =[0, 1]$ be equipped with the Euclidean metric. Consider a sequence $\lbrace y_{n}\rbrace_{n\in \mathbb{N}^{+}}$, where $y_{1} = \frac{3}{4}$ and $y_{n} = \frac{y_{n-1}+1}{2}$, for all $n\geq 2$. For each $n\in \mathbb{N}^{+}$, consider piecewise linear maps $f_{n}\in UE(X)$ defined as follows: 
\begin{align*}
f_{n}(0) =  0,\hspace*{0.1cm} f_{n}(y_{n}) = y_{n+1},\hspace*{0.1cm} f_{n}(1) = 1 
\end{align*}
It is easy to check that, $E(f_{n}) = \lbrace 0, 1\rbrace$, $E^{+}(f_{n}) = \lbrace 0 \rbrace = Se(f)$ and $Sh(f_{n}) = X$, for each $n\in \mathbb{N}^{+}$. Since $f_{n}\xrightarrow[]{uc} I_{X}$, where $I_{X}$ is the identity map on $X$, the nature of expansivity, positive expansivity, sensitivity and shadowing of a point can not be transfer to uniform limits. 
\label{E3.3}
\end{Example}

\begin{Example}
Let $X =[0, 1]$ be equipped with the Euclidean metric. For each $n\in \mathbb{N}^{+}$, consider piecewise linear maps $f_{n}\in UE(X)$ and $f\in UE(X)$ defined as follows: 
\begin{align*}
f_{n}(0) =  0,\hspace*{0.1cm} f_{n}(\frac{1}{n+1}) = \frac{1}{n+1},\hspace*{0.1cm} f_{n}(\frac{3}{4}) = \frac{7}{8},\hspace*{0.1cm} f_{n}(1) = 1 
\end{align*}
\begin{align*}
f(0) = 0, \hspace*{0.1cm}f(\frac{3}{4}) = \frac{7}{8}, \hspace*{0.1cm}f(1) =1
\end{align*}
It is easy to check that, $E(f_{n}) = \lbrace 1\rbrace$, $E^{+}(f_{n}) = \phi$ and $P(f_{n}) = [0, \frac{1}{n+1}]$, for each $n\in \mathbb{N}^{+}$. Since $f_{n}\xrightarrow[]{uc} f$, $E(f) = \lbrace 0, 1\rbrace$, $E^{+}(f) = \lbrace 0\rbrace$ and $P(f) = \phi$. Sequence of functions $f_{n}$ are neither expansive at $x = 0$ nor positively expansive at $x = 1$ but its limit $f$ is expansive at $x = 0$ and positively expansive at $x = 1$. Also, $0\notin Sh(f_{n})$, for each $n\in \mathbb{N}$ but $0\in Sh(f)$. Thus, a point which is not an expansive (positively expansive, shadowable) point of a sequence of functions can be an expansive (positively expansive, shadowable) point of its uniform limit. Since $0\in P(f_{n})$, for all $n\geq 1$ but $0\notin P(f)$, the nature of denseness of periodic points in a neighbourhood of a point can not be transfer to uniform limits.
\label{EE3.3}
\end{Example}

\begin{Example}
Let $X = [0, 1]$ be equipped with the Euclidean metric. For each $k\in \mathbb{N}^{+}$, $P_{k}$ denotes the mid-point of $[0,\frac{1}{k+1}]$ and $Q_{k}$ denotes the mid point of $[0, P_{k}]$. For each $n\in \mathbb{N}^{+}$, consider piecewise linear maps $f_{n}\in UE(X)$  and $f\in UE(X)$ defined as follows: 
\begin{align*}
f_{n}(0) = 0,\hspace*{0.1cm} f_{n}(P_{n}) = Q_{n},\hspace*{0.1cm} f_{n}(\frac{1}{n+1}) = \frac{1}{n+1},\hspace*{0.1cm} f_{n}(\frac{3}{4}) = \frac{7}{8},\hspace*{0.1cm} f_{n}(1) = 1
\end{align*}
\begin{align*}
f(0) = 0 ,\hspace*{0.1cm} f(\frac{3}{4}) = \frac{7}{8}, \hspace*{0.1cm}f(1) =1
\end{align*}
It is easy to check that, $Se(f_{n}) = \lbrace \frac{1}{n+1}\rbrace$ and $Se(f) = \lbrace 0\rbrace $, for each $n\in \mathbb{N}^{+}$. Clearly, $f_{n}\xrightarrow[]{uc} f$. Sequence of functions $f_{n}$ is not sensitive at $x = 0$ but its uniform limit $f$ is sensitive at $x = 0$. Thus, a point which is not a sensitive point of a sequence of functions can be a sensitive point of its uniform limit. 
\label{E3.4}
\end{Example}

\begin{Example}
Let $\alpha_{n}$ be a strictly increasing sequence of positive irrationals converges to $1$ and let $\beta_{n}$ be a  strictly increasing sequence of positive rationals converges to $\frac{1}{\sqrt{2}}$. Define $f_{n}, g_{n}$ and $g$ on $\mathbb{S}^{1}$ by $f_{n}(x) = xe^{2\pi i \alpha_{n}}$, $g_{n}(x) = xe^{2\pi i \beta_{n}}$ and $g(x) = xe^{2\pi i\frac{1}{\sqrt{2}}}$ for each $x\in \mathbb{S}^{1}$ and for each $n\in \mathbb{N}^{+}$. Clearly, $f_{n}\xrightarrow[]{uc} I_{\mathbb{S}^{1}}$, where $I_{S^{1}}$ is the identity map on $\mathbb{S}^{1}$ and $g_{n}\xrightarrow[]{uc} g$. Note that, $Tt(f_{n}) = X = Tt(g)$ and $Tt(f) = \phi = Tt(g_{n})$, for each $n\in \mathbb{N}^{+}$. Therefore, transitivity and mixing nature of a point can not be transfer to uniform limits and a point which is not a transitive (mixing) point of a sequence of functions can be a transitive (mixing) point of its uniform limit.
\label{E3.5}
\end{Example}

\begin{Example}
Let $X = [0, 1]$ be equipped with the Euclidean metric. For each $n\in \mathbb{N}^{+}$, consider piecewise linear maps $f_{n}\in UC(X)$  and $f\in UE(X)$ defined as follows:
\begin{align*}
f_{n}(0) = 0,\hspace*{0.1cm} f_{n}(\frac{1}{2}) = \frac{1}{4},\hspace*{0.1cm} f_{n}(1) = \frac{n}{n+1}
\end{align*}
\begin{align*}
f(0) = 0 ,\hspace*{0.1cm} f(\frac{1}{2}) = \frac{1}{4}, \hspace*{0.1cm}f(1) =1
\end{align*}
Clearly, $f_{n}\xrightarrow[]{uc} f$. Since $Tt(f_{n}) = \phi = Tm(f_{n})$ and $Tt(f) = \lbrace 1 \rbrace = Tm(f)$, a point which is not a transitive (mixing) point of a sequence of functions can be a transitive (mixing) point of its uniform limit.
\label{E3.6}  
\end{Example}

\begin{Example}
Let $X = [0, 1]$ be equipped with the Euclidean metric. For each $n\in \mathbb{N}^{+}$, consider piecewise linear maps $f_{n}\in UC(X)$ defined as follows: 
\begin{align*}
f_{n}(0) = 0,\hspace*{0.1cm} f_{n}(\frac{n}{n+1}) = 1,\hspace*{0.1cm} f_{n}(1) = 1
\end{align*}
Clearly $f_{n}\xrightarrow[]{uc} I_{X}$, where $I_{X}$ is the identity map on $X$. Using Theorem \ref{T2.3}, we get $Sh^{+}(f_{n}) = X$ but $Sh^{+}(I_{X}) = \phi$. Hence, positive shadowable nature of a point can not be transfer to uniform limits.
\label{E3.9}  
\end{Example}

\begin{Example}
Let $X = [0, 1]$ be equipped with the Euclidean metric. For each $n\in \mathbb{N}^{+}$, consider piecewise linear maps $f_{n}\in UC(X)$ and $f\in UE(X)$ defined as follows: 
\begin{align*}
f_{n}(0) =\frac{1}{2(n+1)},\hspace*{0.1cm} f_{n}(\frac{1}{n+1}) = \frac{1}{n+1},\hspace*{0.1cm} f_{n}(\frac{3}{4}) = \frac{7}{8},\hspace*{0.1cm} f_{n}(1) = 1
\end{align*}
\begin{align*}
f(0) = 0 ,\hspace*{0.1cm} f(\frac{3}{4}) = \frac{7}{8}, \hspace*{0.1cm} f(1) = 1
\end{align*}
Clearly, $f_{n}\xrightarrow[]{uc} f$. From \cite[Example 4.2]{CLS} and Theorem \ref{T2.3}, we get $Sh^{+}(f_{n}) = \phi$ and $Sh^{+}(f) = X$. Thus, a point which is not a positive shadowable point of a sequence of functions can be a positive shadowable point of its uniform limit.
\label{E3.10}  
\end{Example}

\begin{Proposition}
Let $f\in UC(X)$ be surjective map on an unbounded metric space $X$. If $f$ is equicontinuous, then $Sp(f) = \phi$.
\label{T3.7} 
\end{Proposition}
\begin{proof}
Let $\epsilon > 0$. Being an unbounded metric space, $X$ can not be cover by finitely many balls of radius $\epsilon$. By equicontinuity of $f$, choose $0 < \delta < \epsilon$ such that $d(x, y) < \delta$ implies $d(f^{n}(x), f^{n}(y)) < \epsilon$, for all $x, y\in X$ and for each $n\in \mathbb{N}$. Let $x\in Sp(f)$ and fix $N = M_{x}^{f}(\delta)\in \mathbb{N}^{+}$. Choose a sequence $0=a_{1} = b_{1} < a_{2} = N =b_{2}$ and $y_{N}\in X$ satisfying $d(f^{N}(x), f^{N}(y_{N})) > 2\epsilon$. By specification at $x$, there exists a $z\in X$ satisfying $d(x, z) < \delta$ and $d(f^{N}(z), f^{N}(y_{N})) < \delta$. Hence $d(f^{N}(x), f^{N}(y_{N})) < 2\epsilon$, a contradiction.
\end{proof}

\begin{Example}
Let $X = \mathbb{R}$ and $Y = [0, 1)$ be equipped with the Euclidean metric. Define sequence of maps $g_{n}: X\rightarrow X$ by $g_{n} = \frac{1}{2n}x$, for each $x\in X$ and for each $n\in \mathbb{N}^{+}$. Define $h_{n}: Y\rightarrow Y$ by $h_{n}(y) = y^{n}$, for each $y\in Y$ and for each $n\in \mathbb{N}^{+}$.
\begin{enumerate}
\item Clearly, $g_{n}\xrightarrow[]{pc} g$ where $g(x) = 0$, for each $x\in X$ and $h_{n}\xrightarrow[]{pc} h$, where $h(y) = 0$, for each $y\in Y$.
\item From Proposition \ref{T3.7}, $Sp(g_{n}) = \phi$, for each $n\in \mathbb{N}^{+}$ but $Sp(g) = X$. Thus, a point which is not a specification point of a sequence of functions can be a specification point of its pointwise limit. 
\item Since $Sp(g) = X$, we can not drop surjectivity in Proposition \ref{T3.7}.
\item Since $Tm(h_{n}) =\phi$, for all $n\geq 0$ and every specification point of a continuous surjective map $f$ on $X$ is a topologically mixing point \cite[Theorem 4.5]{DKDM}, we get $Sp(h_{n}) =\phi$, for all $n\geq 0$. Since $h$ has specification property, even on a bounded space, a point which is not a specification point of a sequence of functions can be a specification point of its pointwise limit.
\end{enumerate}
\label{E3.8}
\end{Example}

\begin{Proposition}
Let $X = \mathbb{R}$ be equipped with the Euclidean metric. For every $f\in UE(X)$, $Wts(f)\subset Sh(f)\subset P_{\alpha}(f)$. 
\label{T3.11}
\end{Proposition}
\begin{proof}
Using the fact that $X$ is a Mandelkern locally compact metric space without isolated points, one can follow steps as in the proofs of \cite[Lemma 8]{WO}, \cite[Lemma 9]{WO} and \cite[Lemma 10]{WO} to prove statements (i), (ii) and (iii) respectively:
\begin{enumerate}
\item[(i)] Suppose that $f$ has the following property: for each $\epsilon > 0$, there exists a $\delta > 0$, such that if every finite sequence $\lbrace x_{0}$. . . $x_{k}\rbrace$ of points of $X$ satisfies $d(T(x_{n}, x_{n+l}) < \delta$, for all $0\leq n\leq k-1$, then there exists a $x\in X$ with $d(T^{n}(x), x_{n}) < \epsilon$, for all $0\leq n \leq k-1$. Then, $f$ has the shadowing property.
\item[(ii)] Let $k\geq 0$ be an integer. Let $\alpha > 0$ and $\eta > 0$ be given. Then for any set of points $\lbrace x_0, x_l,$. . . $,x_k\rbrace$ with $d(T(x_i),x_{i+1}) < \alpha$ for all $0\leq i\leq k-1$, there exists a set of points $\lbrace x'_{0},x'_{1},$. . .$x'_{i}\rbrace$ such that $d(x_i,x'_{i}) < \eta$ for all $0\leq  i \leq k$, $d(T(x'_{i}),x'_{i+1}) < 2\alpha$  for all $0\leq i \leq k-1$ and $x'_{i}\neq x'_{j}$, if $i\neq j$ for all $0\leq i\leq k$ and for all $0\leq j \leq k$.
\item[(iii)] For any finite collection $\lbrace (p_{i},q_{i})\in X\times X : i = 1,...,m\rbrace$ specified together with $0 < \delta < \frac{1}{2\pi}$ such that $d(p_{i},q_{i}) < \delta$, for all $1\leq i\leq m$, $p_{i}\neq p_{j}$ and $q_{i}\neq q_{j}$ whenever $i\neq j$, there exists a uniform equivalence $f$ on $X$ such that $D(f,id) < 2\pi \delta$ and $f(p_{i}) = q_{i}$, for all $1\leq i\leq m$. 
\end{enumerate}
Using similar steps as in the proof of Theorem \ref{T2.6}, one can show that $Wts(f) \subset Sh(f)$. One can follow steps as in the proof of Theorem \ref{T2.7} to obtain the last inclusion.
\end{proof}

\begin{Example}
Let $X = \mathbb{R}$ be equipped with the Euclidean metric. Define a sequence of maps $f_{n} = \frac{n+2}{n+1}x$, for each $x\in X$ and for each $n\in \mathbb{N}^{+}$. Clearly, $f_{n}\xrightarrow[]{pc} I_{X}$, the identity map on $X$. From Proposition \ref{T3.11} and Lemma \ref{T3.12}, we get that $Sh(f_{n}) = Ts(f_{n}) = X = Wts(f_{n}) = P_{\alpha}(f_{n})$, for each $n\in \mathbb{N}^{+}$. Since $Sh(Id_{X}) = Ts(Id_{X}) = \phi = Wts(Id_{X}) = P_{\alpha}(Id_{X})$, topologically stable, weak topologically stable and $\alpha$-persistent nature of a point can not be transfer to pointwise limits.
\label{E3.13}
\end{Example}

\textbf{Acknowledgements:} The first author is supported by CSIR-Junior Research Fellowship (File No.-09/045(1558)/ 2018-EMR-I) of  Government of India.

\end{document}